\documentclass[oneside,10pt]{amsart}          
\usepackage[b5paper]{geometry}	    
\usepackage{amsfonts,amsmath,latexsym,amssymb} 
\usepackage{amsthm}                
\usepackage{mathrsfs,upref}         
\usepackage{mathptmx}		    

\newtheorem{theorem}{Theorem}           
\newtheorem{lemma}{Lemma}               
\newtheorem{corollary}{Corollary}

\theoremstyle{definition}

\newtheorem{definition}{Definition}
\newtheorem{example}{Example}

\newtheorem{remark}{Remark}

\numberwithin{equation}{section}

\begin{document}

\begin{abstract}
      We prove a Leibniz-type inequality for the spread of (real-valued) random variables in terms of their $L^p$-norms.
      The result is motivated by the Kato--Ponce inequality and Rieffel's Leibniz property. 
\end{abstract}

\title{On the Leibniz rule for random variables}
\author{Zolt\'an L\'eka}

\address{Royal Holloway, University of London \\ Egham Hill \\ Egham \\ Surrey \\ TW20 0EX \\ United Kingdom }
\email{zoltan.leka@rhul.ac.uk}

\thanks{This study was supported by the Marie Curie IF Fellowship, Project 'Moments', Grant no. 653943 and by the Hungarian Scientific Research Fund (OTKA) grant
no. K104206.}


\maketitle

\section{Introduction}
 
   For differentiable functions on the real line the Leibniz rule and H\"older's inequality 
   provides us with a simple way to have various estimates of the $L^p$ norms of derivatives of products.
   The Leibniz inequality and its variants have recently appeared in M. Rieffel's fundamental work on the theory of quantum metric spaces, see e.g. \cite{R0, RL, R1, R2}.
   His study was motivated by an urgent need for
   a non-commutative (quantum) analogue of the Gromov--Hausdorff distance between compact metric spaces. 
   In \cite{La2}, \cite{La1} a novel approach to the long-standing problem of 
   finding a proper metric between compact $C^*$-metric spaces has been offered by F. Latr\'{e}moli\`{e}re.
   For a thorough survey on these types of results, we refer the reader to \cite{La}.
   
   We recall that a seminorm $L$ defined on a unital normed algebra $(\mathcal{A}, \|\cdot \| ,{\bf 1}_\mathcal{A})$ 
is strongly Leibniz if
\begin{itemize}
   \item[(i)] $L({\bf 1}_\mathcal{A}) = 0,$
   \item[(ii)](Leibniz inequality) $L(ab) \leq \|a\| L(b) + \|b\| L(a)$ for all $ a,b \in \mathcal{A},$ 
   \item[(iii)] $L(a^{-1}) \leq \|a^{-1}\|^2 L(a)$ whenever  $a \in \mathcal{A}$ is invertible.
\end{itemize}
   One of the simplest (but non-trivial) example of such seminorms is the standard deviation
   defined in ordinary and non-commutative probability spaces as well, see \cite{R2}. 
   Interestingly, several examples of strongly Leibniz seminorms can be defined through derivations
   taking its values in Hilbert bimodules or as commutator norms, see e.g. \cite[Proposition 1.1]{R2}, \cite[Proposition 8]{NW}, \cite[Example 11.5]{R1}.
   
   On the other hand we notice that the Leibniz inequality appears
   in the theory of symmetric Dirichlet forms on real $L^2$ function spaces \cite{BH}, \cite{Fuk}. In fact, 
   it is often used to establish that bounded functions in the domain form an algebra, see e.g. \cite[Corollary 3.3.2]{BH}.
   We just recall that the standard deviation is itself a Dirichlet form over a probability space. 
   
   From the viewpoint of differential operators, Leibniz-type rules have intensively been studied in the theory of non-linear PDEs as well.   
   Let $(-\Delta)^\alpha$ be the fractional Laplace operator defined as the Fourier multiplier
     $$ \widehat{(- \Delta)^\alpha f}(\xi) = |\xi|^{2\alpha} \hat{f}(\xi), \quad \xi \in \mathbb{R}^n, $$
   for any $f$ in the Schwartz space $\mathcal{S}(\mathbb{R}^n).$
   Then the Kato--Ponce inequality or the fractional Leibniz rule in its simplest form asserts the following. 
   Let $1 < r, p_1,q_1,p_2,q_2 < \infty$ such that ${1 \over r} = {1 \over p_1} + {1 \over q_1} = {1 \over p_2} + {1 \over q_2}.$
   Given $0 < \alpha \leq 1,$ for all $f,g \in \mathcal{S}(\mathbb{R}^n),$ one has
        $$ \|(-\Delta)^\alpha (fg)\|_r \leq C( \|g\|_{p_1} \|(-\Delta)^\alpha f\|_{q_1} 
        + \|f\|_{p_2} \|(-\Delta)^\alpha g\|_{q_2}),$$
   where $C = C_{n,\alpha, p_1, q_1, p_2, q_2, r} > 0$ is a constant depending only on $(n,\alpha, p_i, q_i, r).$

   The result and related commutator estimates turned up in the fundamental works of Kato--Ponce \cite{KP}, Christ--Weinstein \cite{CW} and Kenig--Ponce--Vega \cite{KPV}.
   Nowadays the Kato--Ponce inequalities and their extensions have an extensive literature, see e.g. \cite{GO}, \cite{MPTT}.
   For a detailed and thorough exposition of these types of results, we refer the reader to \cite{G} and \cite{MS}.
   A heuristic approach to the inequality in \cite{MPTT} briefly says that 
   if $f$ oscillates more rapidly than $g$, then $g$ is essentially constant with respect to $f$
   and so $(-\Delta)^\alpha(fg)$ behaves like $(-\Delta)^\alpha(f)g.$ The similar statement holds
   if $g$ oscillates more rapidly then $f.$ However, the rigorous proof is based on advanced techniques of harmonic analysis.
   
   
   Leibniz-type rules and related bilinear Poincar\'e--Sobolev inequalities have recently been proved in \cite{BMM}, where the interested reader may find further references.
  
   In the spirit of the Kato--Ponce inequality, our goal is to prove a Leibniz-type rule for random variables and their dispersions 
   around the expected values, but with a strict constant which is surprisingly independent of the H\"older exponents. Namely, the main theorem of the paper is the following.
   
   \begin{theorem}
   Let $(\Omega, \mathcal{F}, \mu)$ be a probability space. For any real $f,g \in L^\infty(\Omega, \mu),$ one has
    $$ \|fg - \mathbb{E}(fg)\|_r \leq \|f\|_{p_1} \|g - \mathbb{E}g \|_{q_1} + \|g\|_{p_2} \|f - \mathbb{E}f \|_{q_2}, $$
    where $1 \leq r,p_1,p_2,q_1,q_2 \leq \infty$ and ${1 \over r} = {1 \over p_1} + {1 \over q_1} = {1 \over p_2} + {1 \over q_2}.$
   \end{theorem}
    
    Motivated by Rieffel's work on Leibniz seminorms, particular cases of the above theorem have already been settled in \cite{BeL} and \cite{L}. The proof here is based on discretization. 
    In fact, we approximate the centered quantities by means of special Laplacian matrices and their products with vectors.
    We shall use an elementary decomposition of centered products and, instead of interpolation methods,
    we shall make a careful application of a H\"older-type inequality. Throughout the paper we shall use the concept
    of random variable but the proofs herein are built on convex analysis and majorization theory, and not on probabilistic methods.
  
    Following the parallel between differential operators on products and centered random variables, in the last section of the paper we shall prove a simple 'chain rule' 
    for the $L^p$-norms of compositions with monotone Lipschitz functions. 
   

 \section{Symmetric norms on $\mathbb{R}^n$}
   
    First, we collect a few properties of symmetric norms. Let $x$ be an $n$-dimensional real vector. We write $x^\downarrow$ for the non-increasing rearrangement of $x.$
    We recall that a norm $\|\cdot\|$ on $\mathbb{R}^n$ is symmetric if it is invariant under the permutation of the components and
    their sign changes. For instance, the $\ell_p$ norms given by $\|x\|_p = (\sum_{i=1}^n |x_i|^p)^{1/p}$ $(1 \leq p < \infty)$ and the vector $k$-norms 
    $\|x\|_{(k)} = \sum_{i=1}^k |x_i|^\downarrow$ $(1 \leq k \leq n)$ are symmetric. Basic properties of symmetric norms 
    are the absolute property, $\|x\| = \| |x|\|,$ and the monotonicity, i.e. $\|x\| \leq \|y\|$ if $|x| \leq |y|$ (see \cite[Section 2]{B}).  
    
    We should point out that von Neumann \cite{JvN} proved that there is a one-to-one correspondence between the symmetric norms on $\mathbb{R}^n$ and unitarily invariant norms on the space of $n \times n$ matrices.
    It is also appropriate to mention here that the vector $k$-norms are
    extremal in the following sense: the well--known Ky Fan Dominance Theorem (see \cite{B} or \cite[Chapter 15]{BS}) says that 
    for all symmetric norms $\|\cdot\|$ the inequality $\|x\| \leq \|y\|$ holds if and only if $\|x\|_{(k)} \leq \|y\|_{(k)}$ is satisfied for all
    vector $k$-norms. Note that we may obtain an infinite family of inequalities from a finite one. 
    
    Let us now introduce the concept of weak majorization or submajorization relation denoted by $\prec_w.$ If $x, y$ are real $n$-dimensional vectors then
    $$ x \prec_w y \quad \mbox{ if and only if } \quad \sum_{i=1}^k x_i^\downarrow \leq \sum_{i=1}^k y_i^\downarrow \mbox{ for all } k=1, \hdots, n.  $$  
    Following the previous notations, $|x| \prec_w |y|$ precisely holds if $\|x\|_{(k)} \leq \|y\|_{(k)}$ for all $k = 1, \hdots, n.$
    In order to prove our main theorem, we shall preliminarily need the concept of weighted vector $k$-norms as well. 
    We say that a vector $w = (w_1, w_2, \hdots, w_n)^T \in \mathbb{R}^n$ is a weight vector if its coordinates are positive  
    and decreasingly ordered $w_1 \geq w_2 \geq \hdots \geq w_n > 0.$ Then the weighted vector $k$-norm is given by 
      $$ \|x\|_{(k)}^w = \sum_{i=1}^k w_i |x_i|^\downarrow.$$
    Recently, these norms have appeared in matrix optimization problems related to the Ky Fan $k$-norms \cite{WDST} and
    robust linear optimization with special weights $w_1 = \hdots = w_{n-1} \geq w_n > 0,$ see \cite{BPS}.

    We recall that the dual norm of any norm $\|\cdot\|$ on $\mathbb{R}^n$ can be given by 
    $\|x\|_{*} =  \max $ $\{ \langle x, y \rangle \colon \|y\| \leq 1\},$
    where $\langle \: \cdot \:,  \: \cdot \:  \rangle $ stands for the usual inner product on $\mathbb{R}^n.$ 
    It is simple to see that if $\|\cdot\|$ is symmetric then $\|\cdot\|_*$ has the same property. 
     
    Our first lemma gives an expression for the dual norm of $\|\cdot\|_{(k)}^w.$ The proof here closely follows the proof of \cite[Proposition 2]{BPS}.
      
      \begin{lemma} The dual norm of the weighted vector $k$-norm is given by
        $$ \|x\|_{(k)^*}^w = 
            \max \left\{ {\| x\|_{(1)} \over w_1}, {\| x\|_{(2)} \over w_1 +  w_2}, \hdots, {\|x\|_{(k-1)} \over w_1 + \hdots + w_{k-1} }, {\| x\|_{(n)} \over w_1  + \hdots + w_k} \right \}.$$
      \end{lemma}
      
      \begin{proof}
         The rearrangement inequality (e.g. see \cite[p. 208]{BS}) and the absolute property of the symmetric norm $\|\cdot\|_{(k)}^w$ imply that
          \begin{align*}
            \|x\|_{(k)^*}^w  &= \max \: \{ \langle |x|, y \rangle \colon {\|y\|_{(k)}^w} \leq 1 \} \\
                             &=  \max \: \{ \langle |x|^\downarrow, y \rangle \colon {\|y\|_{(k)}^w} \leq 1 \mbox{ and } 0 \leq  y_1, \hdots, y_n \}. 
          \end{align*} 
          
         On the other hand, it is simple to see that $\|y\|_{(k)}^w$ $(0 \leq y \in \mathbb{R}^n)$ can be written as the optimal solution of the
         linear program:
           \begin{align*}
             & \max \sum_{i=1}^n u_i y_i   \\
             & \mbox{ s.t. } \sum_{i=1}^j u_i \leq \sum_{i=1}^{\min(j,k)} w_i \hspace{.4cm} j = 1, \hdots, n,  \\
             & \hspace{1cm} u_j \geq 0, \hspace{1.4cm} j = 1, \hdots, n. 
           \end{align*}
          From the strong duality of LP problems we may find that
            \begin{align*} 
            \|y\|_{(k)}^w  =& \; \min \sum_{i=1}^n \sum_{j=1}^{\min(i,k)} s_i w_j  \\
             & \mbox{ s.t. }  \sum_{i=j}^n s_i \geq y_j \hspace{.4cm} j = 1, \hdots, n, \\
             &  \hspace{1cm} s_j \geq 0, \hspace{.4cm} j = 1, \hdots, n. 
            \end{align*}
          Thus  $\|y\|_{(k)}^w \leq 1$ ($y \in \mathbb{R}_+^n$) if and only if 
            \begin{align*}
            \sum_{i=1}^n \sum_{j=1}^{\min(i,k)} s_i w_j &\leq 1 \mbox{ and }\sum_{i=j}^n s_i \geq y_j, s_j \geq 0, \: j = 1, \hdots, n 
            \end{align*} 
          is feasible. 
          Hence we get 
            \begin{align*}
             \|x\|_{(k)^*}^w =& \max \sum_{i=1}^n |x_i|^\downarrow y_i \\ 
             & \mbox{ s.t. }  \sum_{i=1}^n \sum_{j=1}^{\min(i,k)} s_i w_j \leq 1, \\
             &  \hspace{1cm} \sum_{i=j}^n s_i \geq y_j \hspace{1cm} j = 1, \hdots, n \\
             &  \hspace{1cm} s_j, y_j \geq 0, \hspace{1cm} j = 1, \hdots, n. 
           \end{align*} 
          Applying again the strong duality of LP problems, one has
            \begin{align*}
             \|x\|_{{(k)}^*}^w =& \min \theta \\ 
             & \mbox{ s.t. } \sum_{i=1}^{\min(j,k)} \theta w_i - \sum_{i=1}^j \alpha_i \geq 0,  \hspace{.6cm} j = 1, \hdots, n, \\
             &  \hspace{1cm} \alpha_j \geq |x_j|^\downarrow \hspace{2.2cm} j = 1, \hdots, n, \\
             &  \hspace{1cm} \theta \geq 0, \alpha_j \geq 0, \hspace{1.6cm} j = 1, \hdots, n. 
           \end{align*} 
           Therefore, the equality
            $$  \|x\|_{{(k)}^*}^w  =  \max \left\{ {\| x\|_{(1)} \over w_1}, {\| x\|_{(2)} \over w_1 +  w_2}, \hdots, {\|x\|_{(k-1)} \over w_1 + \hdots + w_{k-1} }, {\| x\|_{(n)} \over w_1  + \hdots + w_k} \right \}$$ 
            immediately follows.
      \end{proof}
      
      Throughout the paper let $e_i$ stand for the $i$th standard basis element of $\mathbb{R}^n.$
      Now we may get some description of the extreme points of the unit ball $\mathfrak{B}_{(k)}^w := \{ x \in \mathbb{R}^n \colon \|x\|_{(k)}^w \leq 1 \}.$ 
      These points will be denoted by ${\rm ext }  \: \mathfrak{B}_{(k)}^w.$ 
      
      \begin{lemma}
         For the extreme points of $\mathfrak{B}_{(k)}^w,$ one has
          $${\rm ext } \:  \mathfrak{B}_{(k)}^w 
            \subseteq \left\{ \sum_{i \in S} {\pm e_i \over \sum_{j=1}^{\min(k,|S|)} w_j }  \colon S \subseteq [n] \mbox{ and } 1 \leq |S| \leq {k-1} \mbox{ or } |S| = n\right\}.$$
      \end{lemma}
      
      \begin{proof}
         Let $\mathcal{C}$ stand for the points of the right-hand side in the above inclusion, and let us consider its convex
         hull $\mathfrak{B}_\mathcal{C} = \mbox{ conv } \mathcal{C}.$
         Then it is immediate that ${\rm ext } \:  \mathfrak{B}_{\mathcal{C}} \subseteq \mathcal{C}.$ 
         Furthermore, we get
          $$ \mathfrak{B}_{(k)^*}^w = \{ x \in \mathbb{R}^n \colon |\langle x,y\rangle| \leq 1 \mbox{ for } \|y\|_{(k)}^w \leq 1 \} = \{ x \in \mathbb{R}^n \colon |\langle x,y\rangle| \leq 1 \mbox{ for } y \in \mathcal{C} \}.$$
         Indeed, the last inclusion $\supseteq$ is clear. On the other hand, for any $1 \leq j \leq n,$  one has that 
         $\|x\|_{(j)} = \max \{ |\langle x, y \rangle| \colon y_i \in \{0, \pm 1\} \mbox{ and } \|y\|_1 = j \},$  hence Lemma 1 implies the inclusion $\subseteq$ in the last equality. Thus the polar of $\mathfrak{B}_\mathcal{C}$
         equals the unit ball $\mathfrak{B}_{(k)^*}^w.$ The Bipolar Theorem \cite[Theorem 5.5]{BS} now gives that $\mathfrak{B}_\mathcal{C} =  \mathfrak{B}_{(k)}^w,$
         so ${\rm ext } \:  \mathfrak{B}_{(k)}^w \subseteq \mathcal{C}.$
       \end{proof}
      
      The expression in Lemma 2 might be a bit misleading if $k=1.$ Then the set of the right-hand side is given by the case $|S| = n.$
      We also remark that it might happen that the right-hand side of the above inclusion is strictly larger 
      than ${\rm ext } \;  \mathfrak{B}_{(k)}^w.$
      Indeed, if $w$ is the constant $1$ vector, then it is simple to show that
      \begin{equation}
         {\rm ext } \:  \mathfrak{B}_{(k)}^1 =  {1 \over k} \{-1,1\}^n \cup \{\pm e_i\} \qquad (1 < k < n).
      \end{equation}
      Actually, this now follows from the fact that the dual norm of the vector $k$-norm is max$\displaystyle \Bigl(\|\cdot\|_\infty, {\|\cdot\|_1 \over k} \Bigr).$
      
      Applying a completely different method, one can precisely describe the extreme points 
      of the unit ball in $\mathbb{R}^n$ with respect to the weighted vector $k$-norms \cite[Theorem 3.1]{CL} (called $c$-norms in \cite{CL}).
      However, for our purposes the inclusion of Lemma 2 is enough. 
      
      \section{Proof of Theorem 1}
       
     Let $x$ be an $n$-dimensional vector. We introduce an $n \times n$ symmetric matrix $\Theta_x$ associated to $x$ with zero row and column sums  
     defined by
     $$(\Theta_x)_{ij}  = \begin{cases}
                              {1 \over 2n} (x_i + x_j)  & \mbox{if } i \neq j \\
                               - \sum_{k : k \neq i } (\Theta_x)_{ik}  &  \mbox{if } i  = j. 
                              \end{cases}
     $$   
     For any vectors $x$ and $y,$ let $x \otimes y$ denote the rank-$1$ matrix $xy^T.$
     We will need the following simple but useful lemma.
      
    \begin{lemma} Let us consider two vectors $x$ and $y$ in $\mathbb{R}^n.$ For any symmetric norm $\|\cdot\|$ on $\mathbb{R}^n,$ we have     
      $$\|( \Theta_x - n^{-1} {\bf 1}  \otimes x)y \| \leq \| |x|^\downarrow |y|^\downarrow\|. $$ 
    \end{lemma}

    \begin{proof}
       Relying upon the Ky Fan dominance theorem, it is sufficient to prove that
        $$ |(\Theta_x - n^{-1} {\bf 1}  \otimes x)y| \prec_w |x|^\downarrow |y|^\downarrow;$$
       that is, the statement of the lemma holds with the vector $k$-norms.
       
       Now fix $k.$ By a continuity argument, we may assume that the components of $x$ are nonzero.
       Since we are maximizing a convex function over a compact convex polytope, 
       we have 
       $$ \| ( \Theta_x - n^{-1} {\bf 1}  \otimes x)v \|_{(k)} \leq  \| |x|^\downarrow |v|^\downarrow\|_{(k)} = \|v\|_{(k)}^{|x|^\downarrow} \quad  \mbox{ for all } v \in \mathbb{R}^n$$ if and only if 
       $$ \| ( \Theta_x - n^{-1} {\bf 1}  \otimes x)v \|_{(k)} \leq  1 \quad \mbox{ for all } v \in \mbox{ext } \mathfrak{B}_{(k)}^{|x|^\downarrow}.$$
       In view of Lemma 2, this is equivalent to show for all $l \in \{1, \hdots, k-1\} \cup \{n\}$ that 
       $$ \| ( \Theta_x - n^{-1} {\bf 1}  \otimes x)v \|_{(k)} \leq  \|x\|_{(\min(l,k))}$$
       holds for $v \in \{0 , \pm 1\}^n, \|v\|_1 = l .$ Fix $l.$ By homogeneity, we might have $\|x\|_{(\min(l,k))} = 1.$
       If we consider the previous inequality as a convex optimization problem,  
       we need to prove that the objective value of  
       \begin{equation}
                \max_v \; \| ( \Theta_x - n^{-1} {\bf 1}  \otimes x)v \|_{(k)} \quad \mbox{s.t. } v \in \{0 , \pm 1\}^n, \|v\|_1 = l,
       \end{equation}
       is less than or equal to $1.$ (Of course, if $k=1$ then $l=n$ must hold.)
      
       Next, note that the objective value of (3.1) is convex in $x$ if $x$ 
       varies on any face of the unit ball $\mathfrak{B}_{\min(l,k)}^1.$ The extreme points of any face
       are readily extreme points of the unit ball.
       Hence, in view of (2.1), it turns out that it is sufficient to solve (3.1) in the next cases.  
      
       \noindent {\bf Case (a)} $1 \leq l \leq k-1$ and $x = \pm e_{i_0}.$ Let us assume that $v_{i_0} = 1$ (if $v_{i_0} = -1$ the proof is similar).
       Then
       \begin{align*}
          \| ( \Theta_x - n^{-1} {\bf 1}  \otimes x)v \|_{(k)} &\leq \| ( \Theta_x - n^{-1} {\bf 1}  \otimes x)v \|_{(n)} \\
                                                    &=  {1 \over 2n} \sum_{i=1}^n \Bigl| \sum_{j=1}^n ((x_i - x_j)v_j - (x_i+x_j)v_i) \Bigr|\\
                                                    &= {1 \over 2n} \Biggl((n+1) - \sum_{j \neq i_0} v_j\Biggr) +  {1 \over 2n} \sum_{i \neq i_0} (1 + v_i) \\
                                                    &=1.
       \end{align*}
       If $v_{i_0} = 0,$ we get 
        \begin{align*}
          \| ( \Theta_x - n^{-1} {\bf 1}  \otimes x)v \|_{(k)} &\leq  {1 \over 2n} \sum_{i=1}^n \Bigl| \sum_{j=1}^n ((x_i - x_j)v_j - (x_i+x_j)v_i) \Bigr|\\
                                                    &= {1 \over 2n} \Bigl| \sum_{j \neq i_0} v_j\Bigr| +  {1 \over 2n} \sum_{j \neq i_0} |v_j| \\
                                                    &\leq 1.
       \end{align*}
      
      \noindent {\bf Case (b)}  $1 \leq l \leq k-1$ and $\displaystyle x \in {1\over l} \{+1,-1\}^n.$
      Clearly, we do not decrease the objective value of (3.1) if we allow $v$ to be any vector such that $\|v\|_1 = l$.
      Hence it is enough to show that
        \begin{align*}
             \max \; \{ \| ( \Theta_x - n^{-1} {\bf 1}  \otimes x)v \|_{(k)} \colon x \in \{-1 , 1\}^n, v \in \{\pm e_1, \hdots, \pm e_n \} \} = 1.
          \end{align*}
      Notice that the above expression now is independent of $l.$ 
      Let $i_0$ denote the index of the only non-zero component of $v,$ and let $S_{x,-1}, S_{x,+1}$ be the support sets of $x$ taking values $-1$ and $+1,$
      respectively. We can assume that $x_{i_0} = 1.$ Then 
      \begin{align*}
              \| ( \Theta_x - n^{-1} {\bf 1} \otimes x)v \|_{(k)}  &\leq \| ( \Theta_x - n^{-1} {\bf 1}  \otimes x)v \|_{(n)} \\
                                                        &= {1 \over 2n} \sum_{i=1}^n \Bigl| \sum_{j=1}^n ((x_i - x_j)v_j - (x_i+x_j)v_i) \Bigr|\\
                                                        &= {1 \over 2n} \sum_{j \in S_{x,+1}} (x_{i_0} + x_j) + {1 \over 2n} \sum_{j \in S_{x,-1}} (x_{i_0} - x_j)  \\
                                                        &= 1.
          \end{align*}
          
      \noindent {\bf Case (c)}  $l = n.$ We need to check that
        \begin{align}
             \max \; \{ \| ( \Theta_x - n^{-1} {\bf 1}  \otimes x)v \|_{(k)} \colon v \in \{\pm 1 \}^n \} = \|x\|_{(k)}.
          \end{align}
      Looking again at the extreme points of $\mathfrak{B}_{(k)}^1,$ we may assume that  $x = e_{i_0}$ or $x \in {1 \over  k} \{\pm 1\}^n.$
      In the latter case, with $\tilde{x} = kx,$     
      \begin{align*}
              {1 \over k} \|(\Theta_{\tilde{x}} - n^{-1} {\bf 1}  \otimes \tilde{x})v \|_{(k)}  &\leq \| ( \Theta_{\tilde{x}} - n^{-1} {\bf 1}  \otimes \tilde{x})v \|_\infty \\
                                                        &= \max_{1 \leq i \leq n} {1 \over 2n}  \Bigl| \sum_{j=1}^n ((\tilde{x}_i - \tilde{x}_j)v_j - (\tilde{x}_i+\tilde{x}_j)v_i)\Bigr|\\
                                                        &\leq 1,
          \end{align*}          
      because $(\tilde{x}_i-\tilde{x}_j)(\tilde{x}_i+\tilde{x}_j) = 0$ and $|\tilde{x}_i| = |\tilde{x}_j| = 1.$ 
      
      Lastly, if $x = e_{i_0}$ one can apply an argument similar to Case (b). 
      Let $S_{v_{i_0},-1}, S_{v_{i_0},+1}$ be the support set of $v$ taking values $-v_{i_0}$ and $v_{i_0},$
      respectively. Then 
      \begin{align*}
              \| ( \Theta_x - n^{-1} {\bf 1}  \otimes x)v \|_{(k)}  &\leq \| ( \Theta_x - n^{-1} {\bf 1}  \otimes x)v \|_{(n)} \\
                                                        &= {1 \over 2n} \sum_{i=1}^n \Bigl| \sum_{j=1}^n ((v_j - v_i)x_i - (v_i+v_j)x_j) \Bigr|\\
                                                        &= {1 \over 2n} \left| \left(\sum_{j =1}^n (v_j - v_{i_0}) \right) - 2 v_{i_0} \right| + {1 \over 2n} \sum_{i \neq i_0} |v_{i_0} + v_i|  \\
                                                        &\leq {1 \over 2n} \sum_{j =1}^n |v_j - v_{i_0}| + {2 \over 2n} + {1 \over 2n} \sum_{i \neq i_0} |v_{i_0} + v_i| \\
                                                        &= {1 \over n} (|S_{v_{i_0},-1}| + 1 + |S_{v_{i_0},+1}| -1 ) = 1.
          \end{align*}

      Therefore we may conclude that the optimum value of (3.1) is at most $1$ as desired, and the proof is complete. 
      \end{proof}
  
  \begin{proof}[Proof of Theorem 1] First, assume that $\mu$ is the uniform probability measure on $[n] = \{1, \hdots, n\}.$
   It is simple to see that 
       $$ fg - \mathbb{E}(fg) = {1 \over n} \left[\sum_{j=1}^n (f_i g_i - f_j g_j) \right]_{i=1}^n.$$
    Next, we observe that 
  $$ f_i g_i - f_j g_j  = {1 \over 2}(f_i + f_j)(g_i - g_j) +  {1 \over 2}(g_i+g_j)(f_i - f_j).$$
    Hence
     \begin{equation}
        fg - \mathbb{E}(fg) = - \Theta_f g  - \Theta_g f =  - \Theta_f (g - \mathbb{E}g) - \Theta_g (f - \mathbb{E}f),
      \end{equation}
    where the last equality follows from the fact $\Theta_f 1 = \Theta_g 1 = 0.$ 
    
   Let us consider two vectors $x$ and $y$ in $\mathbb{R}^n.$ We claim that the following version of H\"older's inequality holds
     \begin{equation}
        \|\Theta_x (y - \mathbb{E}y)\|_r \leq \|x\|_{p} \|y- \mathbb{E}y\|_q, 
     \end{equation}
   where $ \displaystyle  {1 \over r} = {1 \over p} + {1 \over q  }$ and $r,p,q \geq 1.$
   One then applies the decomposition (3.3) and the proof is straightforward.      
   
   To prove (3.4), let $\mathfrak{X}_0$ denote the subspace of $n$-dimensional vectors of zero mean value:
   $$\mathfrak{X}_0 = \{ x \in \mathbb{R}^n \colon \sum_{i=1}^n x_i = 0\}.$$
   We denote by $q^*$ (and $r^*$) the conjugate exponent of $q$ (and $r$, resp.), as usual.   
   We now verify that the norm of the operator $\Theta_x \colon (\mathfrak{X}_0, \|\cdot\|_q) \rightarrow (\mathbb{R}^n, \|\cdot\|_r)$ is at most $\|x\|_p.$
   By a duality argument, the adjoint  
   \begin{align*}
            \Theta_x^* \colon (\mathbb{R}^n, \|\cdot\|_{r^*}) &\rightarrow (  \mathbb{R}^n/ \mathbb{R}, \|\cdot\|_{q^*}) \\
                        y &\mapsto \Theta_x y + \mathbb{R}
        \end{align*}
   has the same norm as $\Theta_x$ (see e.g. \cite[Proposition 2.3.10]{Pe}), and $1/p + 1/r^* = 1/q^*$ holds.  
   To get an estimate in the quotient space $(\mathbb{R}^n/\mathbb{R}, \|\cdot\|_{q^*}),$ define the constant
     $$ \lambda_y = {1 \over n} \langle x, y\rangle $$
     for any $y \in \mathbb{R}^n.$ 
    Clearly, $\Theta_x y - \lambda_y {\bf 1} = (\Theta_x - n^{-1} {\bf 1}  \otimes x)y. $ 
    Now let us apply Lemma 3 to the symmetric norm $\|\cdot\| = \|\cdot\|_{q^*}$ and H\"older's inequality to conclude
     \begin{align*}
        \inf_{\lambda \in \mathbb{R}} \| \Theta_x y - \lambda {\bf 1} \|_{q^*} \leq \| \Theta_x y - \lambda_y {\bf 1} \|_{q^*}  = \|( \Theta_x - n^{-1} {\bf 1}  \otimes x)y \|_{q^*} &\leq \| |x|^\downarrow |y|^\downarrow\|_{q^*} \\ 
                                             &\leq \|x\|_p \|y\|_{r^*} .
      \end{align*}
    Hence the claim readily follows and the uniform case is proved.
     
    If we have a general probability space $(\Omega, \mathcal{F}, \mu),$ applying a uniform approximation by simple functions, we can
    reduce the problem to finite state spaces. Let us now use the approximation method in \cite[Proposition 2.1]{BeL}, and the finite discrete case
    reduces to the uniform one. We briefly recall the trick. Let $\mu_\mathbb{Q}$ be any probability measure on an $n$-points state space $S_n$ such that 
    every atom carries a rational probability. Let us write $(\mu_\mathbb{Q})_i = r_i/m,$ where $r_i, m \in \mathbb{N},$ for all $1 \leq i \leq n.$  
    The map
      $$\Phi \colon (x_1, \hdots, x_n) \mapsto (\underbrace{x_1, \hdots, x_1}_{r_1}, \hdots, \underbrace{x_n, \hdots, x_n}_{r_n}) $$ 
         is an injective algebra homomorphism from $\mathbb{R}^n$ into $\mathbb{R}^m.$  Furthermore, if $\lambda$ denotes 
         the uniform probability measure on the $m$-points space, we have 
         $\|x\|_{\ell^p(\mu_\mathbb{Q})} = \|\Phi(x)\|_{\ell^p(\lambda)}$ 
         and $\|xy - \mathbb{E}_{\mu_\mathbb{Q}} (xy)\|_{\ell^p(\mu_\mathbb{Q})} = \|\Phi(x)\Phi(y) - \mathbb{E}_\lambda (\Phi(x)\Phi(y)) \|_{\ell^p(\lambda)}$ in 
         the weighted $\ell^p$ spaces for any $x, y \in \mathbb{R}^n$ and $1 \leq p \leq \infty.$ Since any probability measure on the $n$-points space $S_n$ can be approximated by 
         rational probability measures, the proof is complete.
         
    \end{proof}
    
    \begin{remark}
          Let $L$ be the $n\times n$ Laplacian matrix defined by $L = n^{-1 } {\bf 1} \otimes {\bf 1} - I.$
          Then $-Lf = f - \mathbb{E}_\lambda f,$ where $\lambda$ is the uniform probability measure on $[n].$
          Let us consider the derivation $\partial \colon \ell^2_n(\lambda) \rightarrow \ell^2_n(\lambda) \otimes \ell^2_n(\lambda) = \ell^2_{n \times n}(\lambda \otimes \lambda)$ given by
           $$ (\partial f)_{ij} = {f_i - f_j \over \sqrt{2}},$$
          and let the left and right actions be $f (a_{ij}) = (f_i a_{ij})$ and $(a_{ij}) f = (f_j a_{ij})$ on $\ell^2_{n \times n}(\lambda \otimes \lambda).$
          It is easy to see that 
           $$ -L = \partial^* \partial$$
          and from the Leibniz rule that
           \begin{align*}
             -L(fg) = \partial^*(f \partial g) + \partial^*( (\partial f) g).
           \end{align*}
          Moreover, a little computation gives that $\partial^*(f \partial g) = -\Theta_f g$ and $\partial^*((\partial f) g) = -\Theta_g f.$
          Hence we get back decomposition (3.3) through the derivation $\partial$. 
          
          We also note that 
           $$ \partial^*(f \partial g) = -{1 \over 2} (L(fg) - gLf + fLg)$$ holds.
    \end{remark}

  \section{On the chain rule for random variables}

   In this section, we will prove a 'chain rule' for the $L^p$ norm of compositions of bounded random variables with Lipschitz functions.
   
   Let us introduce the fundamental concept of discrete Laplacians. We recall that an $n \times n$ symmetric matrix $L$
   is {\it Laplacian} if it has zero row and column sums
   and all of its off-diagonals are non-negative. A straightforward corollary of the definition is that $-L$ is positive semi-definite. 
  
  Our first proposition gives a conditional norm estimate of $L.$ The proof is an application of 
  the Calder\'on--Mityagin interpolation (see \cite{C}, \cite{Mi} or \cite[Theorem 15.17]{BS}).
  
\begin{theorem} Let $L$ be an $n \times n$ Laplacian.
    For any symmetric norm $\|\cdot\|$ on $\mathbb{R}^n$ and $x \in \mathfrak{X}_0,$ we have 
     $$ \|L x\| \leq n (\max_{i \neq j}  L_{ij})  \|x\| . $$
\end{theorem}

 \begin{proof}
       Let us consider $L$ as a linear operator which maps the normed space $(\mathfrak{X}_0, \|\cdot\|)$ into $(\mathbb{R}^n, \|\cdot\|).$
       It is simple to see that the dual space is 
       $ (\mathfrak{X}_0, \|\cdot\|)^* = ( \mathbb{R}^n, \|\cdot\|_*)/\mathbb{R}, $ where $\|\cdot\|_*$ denotes the dual norm.
       Since $L1 = 0,$ the adjoint of $L$ is
       \begin{align*}
        L^* \colon (\mathbb{R}^n, \|\cdot \|_*) &\rightarrow ( \mathbb{R}^n, \|\cdot\|_*)/\mathbb{R}, \\ v &\mapsto Lv  + \mathbb{R}. 
       \end{align*}
  
       Let us associate a non-negative vector $x_\infty \in \mathbb{R}^n$ to $L$ defined by 
       $$ x_\infty(i) = \max_{1 \leq j, j \neq i \leq n} L_{ij}, \qquad 1 \leq i \leq n.$$ Set the matrix
       $$ \hat{L} = L - x_\infty \otimes 1.$$
       We claim that the estimates hold: 
       $$  \|\hat{L}\|_{1 \rightarrow 1} \leq n \max_{i \neq j} L_{ij} \quad \mbox{ and } \quad \|\hat{L}\|_{\infty \rightarrow \infty} \leq n \max_{i \neq j} L_{ij}.$$
       In fact, $L$ has zero row and column sums hence for all $1 \leq i \leq n$ we have
       \begin{align*}
          \sum_{j=1}^n |\hat{L}_{ij}| &= \Biggl(  x_\infty(i) + \sum_{1 \leq j, j \neq i \leq n} L_{ij}  \Biggr) + \sum_{1 \leq j , j \neq i \leq n} (x_\infty(i) - L_{ij}) \\
                                                  &= nx_\infty(i) \leq  n (\max_{i \neq j}  L_{ij}) .
       \end{align*}
       Similarly, for any $1 \leq j \leq n,$ 
        \begin{align*}
          \sum_{i=1}^n |\hat{L}_{ij}| &= \Biggl( x_\infty(j) +  \sum_{1 \leq i,i \neq j \leq n} L_{ij} \Biggr) + \sum_{1 \leq i, i \neq j \leq n} (x_\infty(i) - L_{ij}) \\
                                                  &= \sum_{i=1}^n x_\infty(i) \leq  n (\max_{i \neq j}  L_{ij}).
       \end{align*}
       This implies the claim.
       
       Since the dual norm $\|\cdot\|_*$ is symmetric, the Calder\'on-Mityagin interpolation readily gives
       that 
          \begin{align}
             \|\hat{L}^T v\|_* \leq n \max_{i \neq j} L_{ij} \|v\|_*, \qquad v \in \mathbb{R}^n.          
          \end{align}
       
       Pick a $v \in \mathbb{R}^n.$ Set $$ \lambda_v = \langle x_\infty, v \rangle. $$
       Then
       \begin{align*}
          \inf_{\lambda \in \mathbb{R}}  \|Lv  - \lambda 1\|_* \leq \|Lv  - \lambda_v 1\|_* &=  \|Lv  - \langle x_\infty, v\rangle 1 \|_*  \\
                                  &= \|\hat{L}^T v\|_*.
       \end{align*}
       But from (4.1) it now follows that  $$\|L^*\| \leq \|\hat{L}^T\|_{\|\;\|_* \rightarrow \|\;\|_*} \leq n \max_{i \neq j} L_{ij}.$$
       Since $\|L|\mathfrak{X}_0\| = \|L^*\|,$ the proof is complete.
    \end{proof}
    
 \begin{remark}
   Conditional properties of Hermitians have turned out to be particularly useful concepts.
   For instance, a theorem proved by Bhatia and Sano \cite{BaS} says that if $\varphi$ is an operator convex function then the related Loewner matrix 
   $$L_\varphi =\left[ \varphi(x_i) - \varphi(x_j) \over x_i - x_j \right]_{i,j},$$ 
   where $x_1, \hdots, x_n$ are distinct points in $(0, \infty),$ must be conditionally negative definite,
   i.e. it is negative definite on the subspace $\mathfrak{X}_0.$ We refer the reader to \cite{Bap}, \cite{H} for further interesting examples. 
 \end{remark}
    
 Let $L$ be an $n \times n$ Laplacian matrix. If $\varphi \colon \mathbb{R} \rightarrow \mathbb{R}$ is a concave function, then
 Jensen's inequality implies that $ L \varphi( f) \leq \varphi'(f)Lf $ holds for any $f \in \mathbb{R}^n.$ 
 On the other hand, if we consider the Dirichlet form given by $\mathcal{E}(f) = -\langle f, Lf \rangle, $ 
 one has with any Lipschitz function $\varphi$ that
 $$ \mathcal{E}(\varphi \circ f) \leq {\rm Lip}(\varphi)^2 \mathcal{E}(f),$$ 
 which is usually referred to as the Markovian property of $\mathcal{E},$ see e.g. \cite[p. 5]{Fuk},\cite[p. 42]{K}. 
 From now on Lip$(\varphi)$ denotes the Lipschitz constant of the function $\varphi.$
 Interestingly, if the above inequality holds only for monotone Lipschitz functions then it is valid for any Lipschitzian, see e.g. \cite[Proposition 2.1.3]{K}.
 Specifically, the variance $\mbox{Var}_\mu(x) = \mathbb{E}_\mu(|x-\mathbb{E}_\mu x|^2)$ of any vector $x$ (with respect to the probability measure $\mu$)
 satisfies the inequality 
  \begin{align*}
    \mbox{Var}_\mu(\varphi(x)) &\leq {\rm Lip}(\varphi)^2 \mbox{Var}_\mu(x). 
   \end{align*}
  
\begin{definition}
 If $\varphi \colon \mathbb{R} \rightarrow \mathbb{R}$ is a real function,  
 and  $x_1, \hdots, x_n$ are distinct points in $\mathbb{R},$ we define an $n \times n$ symmetric matrix $\Theta[\underline{x};\varphi]$ with zero row 
 and column sums by 
  $$\Theta_{ij}[x_1, \hdots, x_n;\varphi] = \begin{cases}
                              {\varphi(x_i) - \varphi(x_j) \over x_i - x_j} & \mbox{if } i \neq j \\
                               - \sum_{k : k \neq i } \Theta_{ik}[x_1, \hdots, x_n; \varphi]  &  \mbox{if } i  = j. 
                              \end{cases}
    $$
\end{definition}

 Given the uniform probability measure on the state space $\{1,\hdots, n\},$ the next lemma easily links the vector $\varphi(x) - \mathbb{E}\varphi(x)$ and the matrix $\Theta[\underline{x};\varphi].$

\begin{lemma}
  Let $x = (x_1, \dots, x_n) \in \mathbb{R}^n.$ Then
  $$ -{1 \over n} \Theta[\underline{x};\varphi]\left(x - {1 \over n} \sum_{j=1}^n x_j {\bf 1} \right) =\varphi(x) - {1 \over n} \sum_{j=1}^n\varphi(x_j) {\bf 1}.$$
\end{lemma}

\begin{proof}
  Fix an index $1 \leq k \leq n.$ We get
  \begin{align*}
     \left[-{1 \over n} \Theta[\underline{x};\varphi] x \right]_k 
                                               =& -{1 \over n} \sum_{i = 1}^n {\varphi(x_k) -\varphi(x_i) \over x_k - x_i}  (x_i - x_k) \\
                                               =&  \varphi(x_k) - {1 \over n} \sum_{j=1}^n\varphi(x_j) {\bf 1}.
  \end{align*}
  Since $\Theta[\underline{x};\varphi] {\bf 1} = 0$, the proof is complete.
\end{proof}

 It is clear that if $\varphi \colon \mathbb{R} \rightarrow \mathbb{R}$ is a monotone increasing function
 then $\Theta[\underline{x};\varphi]$ is a Laplacian. Now a straightforward corollary of Theorem 2 is the following. 
    
 \begin{corollary}
     Let $\|\cdot\|$ denote any symmetric norm on $\mathbb{R}^n.$ 
     For any monotone Lipschitz function $\varphi \colon \mathbb{R} \rightarrow \mathbb{R},$
     distinct points $x_1, \hdots, x_n \in \mathbb{R}$ and $u \in \mathfrak{X}_0,$ we have
          $$ \|\Theta[x ;\varphi] u\| \leq n \; {\rm Lip}(\varphi) \|u\|.$$
 \end{corollary}
 
 From Lemma 4 we arrive at the next result. 
    
 \begin{corollary}
     Let $\|\cdot\|$ be a symmetric norm on $\mathbb{R}^n.$ 
     For any monotone Lipschitz function $\varphi \colon \mathbb{R} \rightarrow \mathbb{R}$
     and $f = (f_1, \hdots, f_n) \in \mathbb{R}^n,$
     we have
          $$ \Bigl\|\varphi (f) - {1 \over n} \sum_{i=1}^n \varphi(f_i) {\bf 1} \Bigr\| \leq
           {\rm Lip}(\varphi) \Bigl\|f - {1 \over n} \sum_{i=1}^n f_i {\bf 1} \Bigr\| .$$
 \end{corollary}
 
  Now by means of the approximation described in the proof of Theorem 1, we get the following. 
  
 \begin{theorem}
     Let $(\Omega, \mathcal{F}, \mu)$ be a probability space and fix $1 \leq p \leq \infty.$ 
     For any monotone Lipschitz function $\varphi \colon \mathbb{R} \rightarrow \mathbb{R}$ and real
     $f \in L^\infty(\Omega, \mu),$
     we have
          $$ \|\varphi(f) - \mathbb{E}\varphi(f)\|_p \leq {\rm Lip}(\varphi) \|f  - \mathbb{E}f\|_p .$$
 \end{theorem}
     
  \begin{remark}
    In general, there is no distinction if we have a real Lipschitz function given on the real line or a subset of $\mathbb{R}.$ 
    Applying Zorn's lemma we may find an extension from the smaller set to $\mathbb{R}$ with the same Lipschitz number, see \cite[Theorem 1.5.6]{W}.
  \end{remark}
         
  Unfortunately, the next examples show that one cannot remove the restriction to monotone functions in the previous theorem.
  
\begin{example}
  {\rm Let us consider the probability vector $\mu = (1/36, 3/4, 2/9),$ and let $f = (-0.3, 0.28, 0.38).$
     The function $\varphi,$ defined on sufficiently small open neighborhoods of the components of $f,$ is the map $x \mapsto x^{-1},$ and take any of its extension to $\mathbb{R}$  with the same Lipshitz number.
     Specifically, let $\ell^1_\mu$ be the $\ell^1$-space of functions on a $3$-points space with the weight function $\mu.$
     Then we get $$ \|\varphi(f) - \mathbb{E}_\mu \varphi(f)\|_{\ell^1(\mu)} = \|f^{-1} - \mathbb{E}_\mu f^{-1}\|_{\ell^1(\mu)} = 0.57783, $$ whilst
     $$  \|f^{-1}\|_\infty^2 \|f - \mathbb{E}_\mu f\|_{\ell^1(\mu)}= 0.5417. $$
  }
\end{example}

 Another but piecewise linear example shows the failure of the Markov property in weighted $\ell^1$ spaces.

\begin{example}
    {\rm Let us consider the probability vector $\mu = (1/6, 9/12, 1/12 ).$ 
    Define $\varphi$ on the real line by
     $$ \varphi(x) = \begin{cases}
                    -\left(x + {11 \over 15}\right) &\mbox{ if } x \leq {1 \over 15} \\
                    {3 \over 5}x-{21 \over 25}   &\mbox{ if } x \geq  {1 \over 15}.
                  \end{cases}
     $$
     Clearly, ${\rm Lip}(\varphi)=1.$ Set $f = (-11/15,1/15, 13/15).$ A straightforward calculation
     gives that 
      $$ \|f - \mathbb{E}_\mu f\|_{\ell^1(\mu)} = 0.244... $$ and
     $$  \|\varphi(f) - \mathbb{E}_\mu \varphi(f)\|_{\ell^1(\mu)} = 0.26 .$$
     
     On the other hand, we stress that if $\varphi$ is the square function (restricted to the range of $f$) then
 $$ \|f^2 -  \mathbb{E}f^2\|_p \leq 2 \|f\|_\infty \|f-\mathbb{E}f\|_p,$$ for all $1 \leq p \leq \infty.$

     } 
\end{example}

    Our numerical experiments lead us to suspect that if $2 \leq p \leq \infty$ then one may remove the monotonicity assumption in Theorem 3, but we shall leave open this question.

\begin{remark} In \cite{R2} the strong Leibniz inequality related to the standard deviation was studied in the commutative
and non-commutative context as well. We recall that a seminorm $L$ defined on a unital normed algebra $(\mathcal{A}, \|\cdot\|, {\bf 1}_\mathcal{A})$ 
is strongly Leibniz if it satisfies the Leibniz inequality and $L(a^{-1}) \leq \|a^{-1}\|^2 L(a)$ whenever  $a \in \mathcal{A}$ is invertible.

Now if we choose the algebra $L^\infty(\Omega, \mu)$ over a probability space $(\Omega, \mathcal{F}, \mu)$ 
and define the seminorm $L(f) = \|f - \mathbb{E}f\|_p,$ $1 \leq p \leq \infty,$ the Leibniz inequality is satisfied with real functions.
This was completely established in \cite{L} and the result is included in Theorem 1 in the end-point case $p_1 = p_2 = \infty.$
Unfortunately, the strong Leibniz property, in general, fails in probability spaces, see Example 1 if $p=1$.
However, our numerical simulations suggest that if $2 \leq p \leq \infty$ the strong property may hold as well.

\end{remark}

\end{document}